\documentclass[preprint,12pt]{elsarticle}
\title{Conditions on the regularity of balanced $c$-partite tournaments for the existence of strong subtournaments with high minimum degree\footnote{Supported by   PAPIIT-M\'exico under project IN107218;  CONACYT under projects A1-S-12891 and 47510664.}}

\makeatletter
\def\ps@pprintTitle{%
 \let\@oddhead\@empty
 \let\@evenhead\@empty
 \def\@oddfoot{}%
 \let\@evenfoot\@oddfoot}

\usepackage{geometry} 
\usepackage{graphicx}
\usepackage{color}
\usepackage{amsfonts}
\usepackage{amsmath}
\usepackage{booktabs} 
\usepackage{array} 
\usepackage{paralist} 
\usepackage{verbatim} 
\usepackage{subfig} 

\def\b.#1{{\bf #1}}  

\def\cuadrito{ ${\vcenter{\vbox{\hrule height.4pt
                                \hbox{\vrule width.4pt height3pt \kern3pt
                                      \vrule width.4pt}
                                 \hrule height.4pt}}}$ }

\def\cua{ ${\vcenter{\vbox{\hrule height .4pt
                                \hbox{\vrule width .4 pt height7pt \kern7pt
                                      \vrule width .4pt}
                                 \hrule height .4pt}}}$ }

 \newtheorem{theorem}{Theorem}

 \newtheorem{claim}{Claim}

 \newtheorem{corollary}[theorem]{Corollary}

 \newtheorem{lemma}{Lemma}

 \newtheorem{remark}{Remark}

 \newenvironment{proof}[1][Proof]{\textbf{#1.} }{\ \rule{0.5em}{0.5em}\vskip 12pt}
\usepackage{color}

\newif\ifpdf
\ifx\pdfoutput\undefined
\pdffalse 
\else
\pdfoutput=1 
\pdftrue
\fi

\begin{document}

\begin{frontmatter}

\author[label1]{Ana Paulina Figueroa} 
\address[label1]{Departamento de Matem\'aticas ITAM, M\'exico}
\ead{apaulinafg@gmail.com}

\author[label2]{Juan Jos\'e Montellano-Ballesteros\corref{cor1}}
\address[label2]{Instituto de Matem\'aticas, UNAM, M\'exico}
\ead{juancho@math.unam.mx}
\cortext[cor1]{Corresponding author}

\author[label3]{Mika Olsen}
\address[label3]{Departamento de Matem\'aticas Aplicadas y Sistemas, UAM-C, M\'exico}
\ead{olsen@correo.cua.uam.mx }

\begin{abstract} We consider the following problem posed by Volkmann in 2007:    {\it How close to regular must a c-partite tournament be, to secure a strongly
connected subtournament of order $c$?}  We give sufficient conditions on the regularity of balanced  $c$-partite tournaments to assure the existence of strong maximal subtournament  with minimum degree at least  $\left\lfloor \frac{c-2}{4}\right\rfloor+1$. We obtain this result as an application of counting the number of subtournaments of order $c$ for which a vertex has minimum out-degree (resp. in-degree) at most $q\geq 0$.
\end{abstract}


\end{frontmatter}

\section{Introduction}
Let $c$ be a non-negative integer, a  {\it $c$-partite or multipartite tournament} is a digraph obtained from a complete $c$-partite graph by orienting each edge. In 1999  Volkmann  \cite{sub} developed the first contributions in the study of the structure of the strongly connected subtournaments in multipartite tournaments. He proved that every almost regular $c$-partite tournament contains a strongly connected subtournament of order $p$ for each $p\in\{3,4, \ldots, c-1 \} $. In the same paper he also proved that if each partite set of an almost regular  $c$-partite tournament has at least $\frac{3c}{2}-6$ vertices, then there exist a strong subtournament of order $c$. In 2008  Volkmann and Winsen \cite{Almost-regular} proved that every almost regular $c$-partite tournament has a strongly connected subtournament of order $c$ for $c\geq 5$. In 2011 Xu et al.  \cite{X}  proved that every  vertex of regular  $c$-partite tournament with  $c\ge16$, is contained in a strong subtournament of order  $p$ for every  $p\in \{3,4,\dots,c\}$.  The following problem was posed by Volkmann \cite{survey}:  {\it Determine further sufficient conditions for (strongly connected) $c$-partite tournaments to contain a strong
subtournament of order $p$, for some $4 \le p \le c$. How close to regular must a c-partite tournament be, to secure a strongly
connected subtournament of order $c$?}

On  this direction, in 2016 \cite{amo},  we proved that for every (not necessarily strongly connected) balanced $c$-partite tournament of order $n\ge6$, if  the global irregularity of $T$ is at most $\frac{c}{\sqrt{3c+26}}$, then $T$ contains a strongly connected tournament of order $c$.  A    $c$-partite tournament  is  balanced if each of its partite sets have  the same amount of vertices.

In this paper, we consider Volkmann's problem for balanced $c$-partite tournaments. We give sufficient conditions on its regularity to assure the existence of a strong subtournament  with minimum degree at least  $\left\lfloor \frac{c-2}{4}\right\rfloor+1$. We obtain this result as an application of counting the number of subtournaments of order $c$ for which a vertex has minimum out-degree (resp. in-degree) at most $q\geq 0$.

 \section{Notation and definitions}

\label{sec1}

We follow all the definitions and notation of \cite{B-J-G}.  Let $G$ be a $c$-partite tournament of order $n$ with partite sets $\{V_i\}_{i=1}^c$. We call $G$ \textit{balanced}, if all the partite sets have the same number of vertices and we denote by $G_{r,c}$ a balanced $c$-partite tournament satisfying that $|V_i|=r$ for every $ i\in [c]$, where $[c]= \{1, \dots, c\}$. Throughout this paper $|V_i|=r$ for each $i\in [c]$.

Let $G$ be a $c$-partite tournament.  For  $x\in V(G)$ and  $i\in [c]$, the {\it out-neighborhood of $x$ in $V_i$} is $N^+_i(x) =V_i\cap N^+(x)$; the {\it in-neighborhood of $x$ in $V_i$} is $N^-_i(x) =V_i\cap N^-(x)$;   $d^+_i(x) =|N^+_i(x) |$ and   $d^-_i(x) = |N^-_i(x)|$.  

For an oriented graph $D$, the {\it global irregularity of $D$} is defined as $$i_g(D) = \max\limits_{x,y\in V(D)}\{\max\{d^+(x), d^-(x)\}- \min\{d^+(y), d^-(y)\}\}.$$ 
If $i_g(D)=0$ ($i_g(D) \leq 1$, resp.) $D$ is regular (almost regular, resp.).  For our study we introduce another irregularity parameter, namely  {\it  local partite irregularity of $D$},  which is defined as 
$$\mu(D)= \max\limits_{x\in V(D)}\{ \max\limits_{i\in[c]}\{|d_i^+(x)-d_i^-(x)|\}\}. $$
Observe that, for a balanced $c$-partite tournament $G_{r,c}$, $\mu(G_{r,c})\geq \frac{i_g(G_{r,c})}{c-1}$.

\section{Maximal tournaments for which a vertex has  minimum degree at most q }  

The aim of this section is to give sufficient conditions on the minimum degree, local partite  irregularity and global irregularity to obtain a bound on the number of maximal tournaments in a balanced $c$-partite tournament $G_{r,c}$ in which a given vertex  $x \in V(G_{r,c})$ has out-degree (in-degree  resp.) at most $q$, for some given $q\geq 0$.

Let  $x\in V_c$ and let $\mathcal{H}^+_k(x)$ be the set of vectors  $(h_1,h_2, \ldots h_{c-1})\in \{0,1\}^{c-1}$ such that  $h_i=1$ if $d_i^+(x)=r$; $h_i=0$  if $d_i^+(x)=0$; and $\sum_{i=1}^{c-1} h_i=k$. 
A maximal tournament containing the vertex $x$ with out-degree $k$ can be constructed choosing a vertex for each part $V_i$ for $i\in [c-1]$ as follows: Given $\textbf{h}=(h_1,h_2,\dots,h_{c-1})\in \mathcal{H}^+_k(x)$, we choose an out neighbor of $x$ from $V_i$ if and only if $h_i=1$. 
Since the number of maximal tournaments constructed in this way for a fixed $\textbf{h}$ is $\prod_{i=1}^{c-1}d_i^+(x)^{h_i}d_i^-(x)^{1-h_i} $,  
 we have the following remark.

\begin{remark}\label{cuantos}
Let $G_{r,c}$ be a balanced $c$-partite tournament and let $x\in V_c$. The number of maximal tournaments of $G_{r,c}$ for which $x$ has out-degree   $k$ is equal to 
$$\sum_{h\in \mathcal{H}^+_k(x)}\prod_{i=1}^{c-1}d_i^+(x)^{h_i}d_i^-(x)^{1-h_i}.$$   
\end{remark}



Let $x \in V(G_{r,c})$. For each $q\geq 0$, let $T^+_q(x)$ (resp. $T^-_q(x)$)  be the number of maximal tournaments of $G_{r,c}$ for which $x$ has out-degree (resp. in-degree) at most $q$. 
All the following results regarding  $T_q^+(x)$ can be obtained for $T_q^-(x)$ in an analogous way.

Let $x \in V(G_{r,c})$. Assume w.l.o.g that  $x\in V_c$.   By Remark \ref{cuantos},  
$$ T_q^+(x) = \sum\limits_{k=0}^q\left(\sum_{h\in \mathcal{H}^+_k(x)}\prod_{i=1}^{c-1} d_i^+(x)^{h_i}d_i^-(x)^{1-h_i}\right).$$


%
%
%


 In order to bound  $T^+_q(x)$, for any integer  $r\geq 2$, and  $g_1,g_2, \ldots,g_{s}$  real numbers such that $0 \le g_i \leq r$, we define
 $$M(g_1, \dots , g_s; k)=\sum_{h\in \mathcal{H}^s_{k}}\prod_{i=1}^{s} g_i^{h_i}(r-g_i)^{1-h_i},$$
 where $\mathcal{H}^s_{k}$ is the set of $s$-vectors  $(h_1,h_2, \ldots,h_s)\in \{0,1\}^s$ such that:  if $g_i=r$ then $h_i=1$;  if $g_i=0$ then  $h_i=0$; and  $\sum\limits_{i=1}^sh_i=k$.


If for a given $x\in V(G_{r,c})$, $g_i= d^+_i(x)$, with $i\in [c-1]$,  the following lemma gives a sufficient condition on the out-degree of $x$ to assure that the number of maximal tournaments in which $x$ has out-degree $q$ is at least equal to the number of maximal tournaments in which $x$ has out-degree $q-1$.

\begin{lemma}\label{chico} Let $r\geq 2$ be an integer, and let $g_1, \dots , g_{s}$ be real numbers such that  $0\leq g_i \leq r$.  Let  $\Gamma= max\{ g_i\}_{i\in[s]}$ and  $\gamma = min \{g_i\}_{i\in[s]}$. If for some integer $q \geq 1$ we have that  $\sum\limits_{i\in [s]} g_i \geq q(r+\Gamma-\gamma) - \Gamma $, then  $$M(g_1,  \dots, g_{s}; q) \geq M(g_1,  \dots, g_{s}; q-1).$$
\end{lemma}
\begin{proof}  Let $g_1,\ldots,g_{s}$ be real numbers such that $0 \le g_i \leq r$, and let $q\geq 1$.  Suppose w.o.l.g. that for every  $i$,  if $1\leq i \leq t$ then $0<g_i<r$; if $t+1 \leq i \leq t+p_r$ then $g_i = r$; and if $t+p_r+1 \leq i \leq s$ then $g_i=0$.  Observe that for every ${\bf h}=(h_1,\ldots. h_{s})\in \mathcal{H}^{s}_{q-1} $ and every ${\bf h'}=(h'_1,\ldots. h'_{s})\in \mathcal{H}^{s}_{q} $ we have that if  $t+1 \leq i \leq t+p_r$, then $h_i=h'_i = 1 $; and if  $t+p_r+1 \leq i \leq s$, then $h_i=h'_i = 0 $.  Notice that if $p_r \geq q$, then  $\mathcal{H}^{s}_{q-1}  = \emptyset$ which implies that $M(g_1,  \dots, g_{s}; q-1)=0$  and  the lemma follows. Thus, we can suppose that $q\geq p_r+1$.

 For each ${\bf h}=(h_1,\ldots. h_{s})\in \mathcal{H}^{s}_{q-1} $  let  
$F({\bf h}) = \{(h'_1,\dots, h'_{s})\in \mathcal{H}^{s}_{q} :  h_i \leq  h'_i  \hbox{ for } i \in [s]\}$  and let $a({\bf h}) = \{ j : h_j = 1 \hbox{ for } j \in [t]\}$. Observe that for every ${\bf h}\in \mathcal{H}^{s}_{q-1}$,  $|a({\bf h})| = q-1-p_r$.

By definition of $\mathcal{H}^{s}_{q} $ and $\mathcal{H}^{s}_{q-1} $, it follows that given ${\bf h}\in \mathcal{H}^{s}_{q-1} $ and ${\bf h' }\in F({\bf h})\subseteq \mathcal{H}^{s}_{q} $ (except for some $j_0\in [t]\setminus a({\bf h})$, where $h'_{j_0}= h_{j_0}+1$), we have that $h_i = h'_i$ for every $i\in [s]\setminus\{j_0\}$.  

Thus, \begin{equation}\label{uno} \frac{\sum\limits_{\textbf{h'}\in F(\textbf{h})} \prod\limits_{i=1}^{s} g_i^{h'_i} (r-g_i)^{1-h'_i}}{  \prod\limits_{i=1}^{s} g_i^{h_i} (r-g_i)^{1-h_i}}= \sum\limits_{\textbf{h'}\in F(\textbf{h})} \frac{\prod\limits_{i=1}^{s} g_i^{h'_i} (r-g_i)^{1-h'_i}}{  \prod\limits_{i=1}^{s} g_i^{h_i} (r-g_i)^{1-h_i}}=\sum\limits_{j\in [t]\setminus a({\bf h})} \frac{g_j}{r-g_j}.\end{equation}

 \begin{claim}\label{cota0}  $\sum\limits_{j\in [t]\setminus a({\bf h})} \frac{g_j}{r-g_j}\geq q-p_r$. 
 \end{claim}
 Suppose that $\sum\limits_{j\in [t]\setminus a({\bf h})} \frac{g_j}{r-g_j} < q-p_r$. Let $Q_0 = \min\limits_{i \in [t]} \{ g_i \}$. Thus, $\sum\limits_{j\in [t]\setminus a({\bf h})} \frac{g_j}{r-Q_0} < q-p_r$ and therefore  $\sum\limits_{j\in [t]\setminus a({\bf h})} g_j < (r-Q_0)( q-p_r)$ . On the other hand,  $\sum\limits_{j\in [s]} g_j  = \sum\limits_{j\in [t]} g_i + rp_r =  \sum\limits_{j\in [t]\setminus  a({\bf h})}  g_i +  \sum\limits_{j\in a({\bf h} )} g_i + rp_r$. Hence,  $ \sum\limits_{j\in [t]\setminus  a({\bf h})}  g_i  = \sum\limits_{j\in [s]} g_j  -  \sum\limits_{j\in a({\bf h} )} g_i - rp_r$ which implies that
  $$  (r-Q_0)( q-p_r) > \sum\limits_{j\in [s]} g_j  -  \sum\limits_{j\in a({\bf h} )} g_i - rp_r$$ and therefore, after some easy calculation, we see that 
   $ rq    +  \sum\limits_{j\in a({\bf h} )} g_i  -Q_0(q -p_r) > \sum\limits_{j\in [s]} g_j .$
 Let $Q_1 = max \{ g_i : i \in [t]\}$. Since $|a({\bf h})| = q-1-p_r$, it follows that  
 $$ rq +  Q_1(q-1-p_r) -Q_0(q -pr) > \sum\limits_{j\in [s]} g_j. $$  Since $ \Gamma\geq Q_1\geq Q_0 \geq \gamma$  and $p_r\geq 0$, we see that 
 $$\begin{array}{lcl}
Q_1(q-1-p_r) -Q_0(q -pr)& \leq & \Gamma (q-1-p_r)  -\gamma(q -pr)\\&= &\Gamma (q-1)  -\gamma q   - p_r(\Gamma- \gamma)\\
&\leq&   \Gamma (q-1)  -\gamma q.
\end{array}  $$
Thus,
 $$rq +  \Gamma(q-1) -\gamma q = q(r+\Gamma- \gamma) - \Gamma > \sum\limits_{j\in [s]} g_j $$ which, by hypothesis, is not possible and the claim follows.

 From Claim  \ref{cota0} and (\ref{uno}) it follows that for each ${\bf h}=(h_1,\ldots. h_{s})\in \mathcal{H}^{s}_{q-1}$
\begin{equation}\label{dos}\sum\limits_{\textbf{h'}\in F({\bf h})}   \prod\limits_{i=1}^{s} g_i^{h'_i} (r-g_i)^{1-h'_i} \geq (q-p_r) \prod\limits_{i=1}^{s} g_i^{h_i} (r-g_i)^{1-h_i}.\end{equation}

Observe that, for every $\textbf{h'}\in \mathcal{H}^{s}_{q} $,  $|\{ j : h'_j = 1 \hbox{ with } j \in [t]\}| = q-p_r$. Therefore,  for every $\textbf{h'}\in \mathcal{H}^{s}_{q} $ there are exactly  $q-p_r$  elements $\textbf{h}\in \mathcal{H}^{s}_{q-1} $ such that, 
$\textbf{h'} \in F(\textbf{h})$. Thus,
$$ \sum\limits_{\textbf{h}\in H^{s}_{q-1}} \left(\sum\limits_{\textbf{h'}\in F(\textbf{h})} \prod\limits_{i=1}^{s} g_i^{h'_i} (r-g_i)^{1-h'_i}\right)=  (q-p_r) \sum\limits_{\textbf{h'}\in H^{s}_{q}}   \prod\limits_{i=1}^{s} g_i^{h'_i} (r-g_i)^{1-h'_i}. $$
On the other hand, by (\ref{dos}) we see that 
 $$ \sum\limits_{\textbf{h}\in H^{s}_{q-1}}\left( \sum\limits_{\textbf{h'}\in F(\textbf{h})} \prod\limits_{i=1}^{s} g_i^{h'_i} (r-g_i)^{1-h'_i}  \right)\geq  \sum\limits_{\textbf{h}\in H^{s}_{q-1}}  (q-p_r) \prod\limits_{i=1}^{s} g_i^{h_i} (r-g_i)^{1-h_i}$$ implying that
$$\sum\limits_{\textbf{h'}\in H^{s}_q}  \prod\limits_{i=1}^{s} g_i^{h'_i} (r-g_i)^{1-h'_i}\geq  \sum\limits_{\textbf{h}\in H^{s}_{q-1}}  \prod\limits_{i=1}^{s} g_i^{h_i} (r-g_i)^{1-h_i}$$ which, by definition,  is equivalent to $M(g_1,  \dots, g_{s}; q) \geq M(g_1,  \dots, g_{s}; q-1)$  and the lemma follows. \end{proof}

\begin{corollary}  Let $r\geq 2$, $c\geq 3$ and $G_{r,c}$ be a balance $c$-partite tournament such that for some $q\geq 1$,  $\delta(G_{r,c}) \geq q \big(r+\mu(G_{r,c}) \big).$ 
Then, for every $x\in V(G_{r,c})$ the number of maximal tournaments in which $x$ has out-degree $q$ is at least equal to the number of maximal tournaments in which $x$ has out-degree $q-1$.
\end{corollary}

The following theorem gives a condition regarding the minimum degree and the local partite irregularity to obtain a bound of $T_q^+(x)$.  

\begin{theorem}\label{general} Let $r\geq 2$, $c\geq 5$ and $G_{r,c}$ be a balance $c$-partite tournament such that for some $q\geq 0$,  $\delta(G_{r,c}) \geq q \big(r+\mu(G_{r,c}) \big)\big(\frac{c-1}{c-2}\big).$ 
Then, for every $x\in V(G_{r,c})$, $$T^+_q(x) \leq \sum\limits_{k=0}^{q} {{c-1}\choose {k}} \Big(\frac{d^+(x)}{c-1}\Big)^k \Big(\frac{d^-(x)}{c-1}\Big)^{c-1-k}. $$
\end{theorem}
\begin{proof} Let $x\in V(G_{r,c})$, and suppose $x\in V_c$.  By Lemma \ref{cuantos},   we see that $$T_q^+(x) = \sum\limits_{k=0}^{ q}\sum_{h\in \mathcal{H}_k(x)}\prod_{i=1}^{c-1}d_i^+(x)^{h_i}d_i^-(x)^{1-h_i} =  \sum\limits_{k=0}^{ q}  M(d^+_1(x),  \dots, d^+_{c-1}(x); k).$$ 
For each $i$, with $i \in [c-1]$, let $g_i = d^+_i(x)$, and assume that  $g_{c-1}=  max\{ g_i\}_{i\in[c-1]}= \Gamma$ and $g_{c-2}=  min \{g_i\}_{i\in[c-1]}= \gamma$. Let  $g'_1, g'_2, \dots , g'_{s-1}, g'_{s}$ be real numbers such that,  for $i\in [c-3]$,  $g'_i = g_i$; and $g'_{c-2} = g'_{c-1} =  \frac{g_{c-2} + g_{c-1}}{2}$. 

 \begin{claim}\label{supremo}   $\sum\limits_{k=0}^{q} M(g_1, \dots, g_{c-1}; k)  \leq \sum\limits_{k=0}^{q} M(g'_1, \dots, g'_{c-1}; k). $
  \end{claim} If $q=0$,   $\sum\limits_{k=0}^{q} M(g_1, \dots, g_{c-1}; 0) = \prod\limits_{i=1}^{c-1} (r-g_i)$. Since $(r-g_{c-2})(r-g_{c-1})\leq (r-\frac{g_{c-2} + g_{c-1}}{2})^2$,  the claim follows.  Assume that $q\geq 1$.   For the sake of  readability, in what follows,  $g_1, \dots, g_{c-1}$ and  $g_1, \dots, g_{c-3}$ can be denoted as $g_{[c-1]}$ and $g_{[c-3]}$, respectively.  Observe that $$M( g_{[c-1]}; 0) = M(g_{[c-3]}; 0)M(g_{c-2}, g_{c-1}; 0);$$  
  $$M( g_{[c-1]}; 1)= M( g_{[c-3]}; 1)M(g_{c-2}, g_{c-1}; 0) +M(g_{[c-3]}; 0)M(g_{c-2}, g_{c-1}; 1) $$ and  for every  $k \geq 2$,  $$M(g_{[c-1]}; k) =\sum_{j=0}^2 M(g_{[c-3]}; k-j)M(g_{c-2}, g_{c-1}; j).$$  
Therefore, for $q=1$, 
$$\begin{array}{lcl}
\sum\limits_{k=0}^{1} M(g_{[c-1]}; k) &=& M(g_{[c-3]}; 0)\left[M(g_{c-2}, g_{c-1}; 0)+M(g_{c-2}, g_{c-1}; 1)\right]\\
&+&  M(g_{[c-3]}; 1)M(g_{c-2}, g_{c-1}; 0);
\end{array}$$
and  for $q\geq 2$, 
$$\begin{array}{lcl}
\sum\limits_{k=0}^{q} M(g_{[c-1]}; k) &=& \sum\limits_{k=0}^{q-2} M(g_{[c-3]}; k)\left[M(g_{c-2}, g_{c-1}; 0)+M(g_{c-2}, g_{c-1}; 1)+M(g_{c-2}, g_{c-1}; 2)\right] \\
& + &M(g_{[c-3]}; q-1)\left[M(g_{c-2}, g_{c-1}; 0)+M(g_{c-2}, g_{c-1}; 1)\right]\\
&+&  M(g_{[c-3]}; q)M(g_{c-2}, g_{c-1}; 0).
\end{array}$$

It is not hard to see that for  any pair of  reals $0\leq x,y\leq r $,  $M(x, y; 0)= (r-x)(r-y)$;  $M(x, y; 1) = r(x+y)  - 2xy$ and $M(x, y; 2)= xy$. Therefore,  $M(x, y; 2) +M(x, y; 1) + M(x, y; 0) = r^2$. Since  for $i\in [c-3]$,  $g'_i = g_i$  and $g_{c-2}+ g_{c-1} = g'_{c-2}+ g'_{c-1}$, we have, after some easy calculations, that 
$$\begin{array}{lcl}
\sum\limits_{k=0}^{q} M(g'_{[c-1]}; k) - \sum\limits_{k=0}^{q} M(g_{[c-1]}; k)&= &M(g_{[c-3]}; q-1)\left[g_{c-2}g_{c-1}- g'_{c-2}g'_{c-1} \right] \\
&+&  M(g_{[c-3]}; q)\left[ g'_{c-2}g'_{c-1} - g_{c-2}g_{c-1} \right]\\
&=& \left( g'_{c-2}g'_{c-1} - g_{c-2}g_{c-1}\right)\left[M(g_{[c-3]}; q) - M(g_{[c-3]}; q-1) \right].
\end{array}  $$

Since  $g'_{c-2}g'_{c-1} \geq g_{c-2}g_{c-1} $, it follows that  $\sum\limits_{k=0}^{q} M(g_{[c-1]}; k)\leq \sum\limits_{k=0}^{q} M(g'_{[c-1]}; k)$   if and only if 
$M(g_{[c-3]}; q-1)\leq M(g_{[c-3]}; q)$.

Since  $\sum\limits_{i\in [c-1]} g_i = d^+(x) \geq \delta(G_{r,c})\geq q \big(r+\mu(G_{r,c}) \big)\big(\frac{c-1}{c-2}\big)$, it follows that $d^+(x) \big(\frac{c-2}{c-1}\big) = d^+(x) - \frac{d^+(x)}{c-1}\geq \big(r+\mu(G_{r,c}) \big)$. Therefore, $d^+(x)\geq  \big(r+\mu(G_{r,c}) \big) + \frac{d^+(x)}{c-1}$.  On the one hand, clearly  $\gamma \leq \frac{d^+(x)}{c-1}$ and by definition $\mu(G_{r,c}) \geq \Gamma- \gamma$. It follows that $d^+(x) = \sum\limits_{i\in [c-1]} g_i \geq q(r+\Gamma-\gamma) + \gamma$. Since $g_{c-1}= \Gamma$ and $g_{c-2}= \gamma$, we see that $\sum\limits_{i\in [c-3]} g_i \geq q(r+\Gamma-\gamma) - \Gamma$.  On the other hand, observe that  $\Gamma\geq  \Gamma^* = max\{ g_i\}_{i\in[c-3]}$ and  $\gamma \leq \gamma^* = min \{g_i\}_{i\in[c-3]}$. Since $q\geq 1$,  it follows that 
$q(r+\Gamma-\gamma) - \Gamma \geq q(r+\Gamma^* - \gamma^*) - \Gamma^*$ which implies that $\sum\limits_{i\in [c-3]} g_i \geq q(r+\Gamma^*-\gamma^*) -\Gamma^*$. Hence, by Lemma \ref{chico},  $M(g_{[c-3]}; q-1)\leq M(g_{[c-3]}; q)$, and from here the claim follows.  
  
Observe that  $\Gamma\geq \Gamma' = max\{ g'_i\}_{i\in[c-1]}$ and  $\gamma \leq \gamma' = min \{g'_i\}_{i\in[c-1]}$. Since     $\sum\limits_{i\in [c-1]} g_i  = \sum\limits_{i\in [c-1]} g'_i$ it follows that  $\sum\limits_{i\in [c-1]} g'_i \geq q(r+\Gamma'-\gamma')\big(\frac{c-1}{c-2}\big)$, and clearly $0\leq g'_i\leq r$. Hence, we can 
iterate this procedure, and by the way  $g'_{c-2}$ and $g'_{c-1}$ are defined, we see that the limit of the difference $\Gamma'- \gamma'$  by iterating this procedure is zero. Thus, by Claim \ref{supremo}, it follows that $T_q^+(x)$ is bounded by $\sum\limits_{k=0}^{q} M(\frac{d^+(x)}{c-1},  \dots, \frac{d^+(x)}{c-1}; k)$. 
Finally, by definition, for each $k\in [q]$,  
$$\begin{array}{lcl}
M\Big(\frac{d^+(x)}{c-1}, \dots , \frac{d^+(x)}{c-1}; k\Big)&=& \sum\limits_{h\in \mathcal{H}^{c-1}_{k}}\prod\limits_{i=1}^{s}  \Big(\frac{d^+(x)}{c-1}\Big)^{h_i}\Big(r-\frac{d^+(x)}{c-1}\Big)^{1-h_i}  \\
& = &\sum\limits_{h\in \mathcal{H}^{c-1}_{k}} \Big(\frac{d^+(x)}{c-1}\Big)^{k}\Big(r-\frac{d^+(x)}{c-1}\Big)^{c-1-k}\\
 &=& {{c-1}\choose {k}}   \Big(\frac{d^+(x)}{c-1}\Big)^{k}\Big(r-\frac{d^+(x)}{c-1}\Big)^{c-1-k}, 
\end{array}$$ and it is not hard to see that $ r-\frac{d^+(x)}{c-1} = \frac{d^-(x)}{c-1}$.  From here the result follows. \end{proof}

The following theorem gives a condition regarding the minimum degree; the local partite irregularity and the global irregularity to obtain a bound of $T_q^+(x)$.  

\begin{theorem}\label{cota2} Let $r\geq 2$, $c\geq 5$ and $G_{r,c}$ be a balance $c$-partite tournament. If  for some $q\geq 0$,  $\delta(G_{r,c}) \geq q \big(r+\mu(G_{r,c}) \big)\big(\frac{c-1}{c-2}\big)$ and 
 $i_g(G_{r,c}) = r(c-1)\beta$ with  $0\leq \beta < \frac{c-2q-2}{c}$, then for every $x\in V(G_{r,c})$ we have that  
$$T_q^+(x) \leq   {{c-1}\choose {q+1}}  \Big(\frac{r}{2}\Big)^{c-1}\frac{ \Big(1+\beta\Big)^{c-2-2q} (q+1) }{c(1-\beta)-2q-2}.$$
\end{theorem}
\begin{proof} Let $x\in V(G_{r,c})$. By Theorem \ref{supremo},  it follows that 
\begin{equation}\label{cota1}T^+_q(x) \leq \sum\limits_{k=0}^{q} {{c-1}\choose {k}} \Big(\frac{d^+(x)}{c-1}\Big)^k \Big(\frac{d^-(x)}{c-1}\Big)^{c-1-k}. \end{equation}   Observe that $\sum\limits_{k=0}^{q} {{c-1}\choose {k}}\Big(\frac{d^+(x)}{c-1}\Big)^k \Big(\frac{d^-(x)}{c-1}\Big)^{c-1-k} = \Big(\frac{d^-(x)}{c-1}\Big)^{c-1} \sum\limits_{k=0}^{q} {{c-1}\choose {k}} \Big(\frac{d^+(x)}{d^-(x)}\Big)^k$.
For every $q$, with $0\leq q \leq c-1$, let $g(q) = \sum\limits_{k=0}^q {{c-1}\choose {k}} \Big(\frac{d^+(x)}{d^-(x)}\Big)^k$.
Observe that for $q< c-1$,
 $$ \begin{array}{lcl} g(q+1) = 1 +  \sum\limits_{k=1}^{q+1} {{c-1}\choose {k}} \Big(\frac{d^+(x)}{d^-(x)}\Big)^k 
&=& 1 +  \sum\limits_{k=0}^{q} {{c-1}\choose {k+1}} \Big(\frac{d^+(x)}{d^-(x)}\Big)^{k+1} \\
& = & 1 +\Big(\frac{d^+(x)}{d^-(x)}\Big)  \sum\limits_{k=0}^{q} {{c-1}\choose {k+1}} \Big(\frac{d^+(x)}{d^-(x)}\Big)^{k} \\
& = & 1 +\Big(\frac{d^+(x)}{d^-(x)}\Big) \sum\limits_{k=0}^{q} {{c-1}\choose {k}}\frac{c-1-k}{k+1} \Big(\frac{d^+(x)}{d^-(x)}\Big)^{k}.\\
\end{array}$$
 Notice that for each $k\leq q$, $\frac{c-1-q}{q+1}\leq \frac{c-1-k}{k+1}$. Thus, 
 $$ \begin{array}{lcl}
g(q+1) &\geq& 1 +\Big(\frac{d^+(x)}{d^-(x)}\Big) \sum\limits_{k=0}^{q} {{c-1}\choose {k}}\frac{c-1-q}{q+1}\Big(\frac{d^+(x)}{d^-(x)}\Big)^{k} \\ 
&=&  1+ \Big(\frac{d^+(x)}{d^-(x)}\Big)\Big(\frac{c-1-q}{q+1}\Big)  \sum\limits_{k=0}^{q} {{c-1}\choose {k}} \Big(\frac{d^+(x)}{d^-(x)}\Big)^{k}  \\
& >& \Big(\frac{d^+(x)}{d^-(x)}\Big)\Big(\frac{c-1-q}{q+1}\Big)g(q).
 \end{array}$$
  On the other hand, $g(q+1)= g(q) + {{c-1}\choose {q+1}} \Big(\frac{d^+(x)}{d^-(x)}\Big)^{q+1}$. Therefore  
  $$ g(q) + {{c-1}\choose {q+1}} \Big(\frac{d^+(x)}{d^-(x)}\Big)^{q+1} >  \Big(\frac{d^+(x)}{d^-(x)}\Big)\Big(\frac{c-1-q}{q+1}\Big)g(q)$$ which implies that 
   \begin{equation}\label{tres} {{c-1}\choose {q+1}}\Big(\frac{d^+(x)}{d^-(x)}\Big)^{q+1} >  \Bigg(\Big(\frac{d^+(x)}{d^-(x)}\Big)\Big(\frac{c-1-q}{q+1}\Big)-1\Bigg)g(q).\end{equation}
  Clearly,     $\frac{d^+(x)}{d^-(x)} \geq \frac{\delta(G_{r,c})}{\Delta(G_{r,c})}$, and since $\Delta(G_{r,c}) = \frac{r(c-1) + i_g(G_{r,s})}{2}$,  $\delta(G_{r,c}) = \frac{r(c-1) -  i_g(G_{r,s})}{2}$, and $i_g(G_{r,c}) = r(c-1)\beta$, it is not hard to see that $\frac{\delta(G_{r,c})}{\Delta(G_{r,c})} =\frac{1-\beta}{1+\beta}$.  Moreover, since $\beta < \frac{c-2q-2}{c}$, it follows that $\frac{1-\beta}{1+\beta}> \frac{2q+2}{2c-2q-2}= \frac{q+1}{c-q-1}$. Therefore $\Big(\frac{1-\beta}{1+\beta}\Big)\Big(\frac{c-1-q}{q+1}\Big)-1>0.$ Thus, $\Big(\frac{d^+(x)}{d^-(x)}\Big)\Big(\frac{c-1-q}{q+1}\Big)-1\geq \Big(\frac{1-\beta}{1+\beta}\Big)\Big(\frac{c-1-q}{q+1}\Big)-1= \frac{(1-\beta)(c-1-q)-(1+\beta)(q+1)}{(1+\beta)(q+1)}= \frac{c(1-\beta)-2q-2}{(1+\beta)(q+1)}>0.$ Hence, by (\ref{tres}), 
$$ {{c-1}\choose {q+1}}\Big(\frac{d^+(x)}{d^-(x)}\Big)^{q+1} >  \Bigg(\frac{c(1-\beta)-2q-2}{(1+\beta)(q+1)}\Bigg)g(q)$$ and then 
$${{c-1}\choose {q+1}}\Big(\frac{d^+(x)}{d^-(x)}\Big)^{q+1}\frac{ (1+\beta)(q+1)}{c(1-\beta)-2q-2} >  g(q).$$
Therefore, it follows that, for $q< c-1$, 
$$ \begin{array}{lcl}
\sum\limits_{k=0}^{q} {{c-1}\choose {k}}\Big(\frac{d^+(x)}{c-1}\Big)^k \Big(\frac{d^-(x)}{c-1}\Big)^{c-1-k} &=& \Big(\frac{d^-(x)}{c-1}\Big)^{c-1} \sum\limits_{k=0}^{q} {{c-1}\choose {k}} \Big(\frac{d^+(x)}{d^-(x)}\Big)^k\\ 
 &= & \Big(\frac{d^-(x)}{c-1}\Big)^{c-1} g(q)\\
& <& \Big(\frac{d^-(x)}{c-1}\Big)^{c-1} {{c-1}\choose {q+1}}\Big(\frac{d^+(x)}{d^-(x)}\Big)^{q+1}\frac{ (1+\beta)(q+1)}{c(1-\beta)-2q-2}. \end{array}$$
Thus, by (\ref{cota1}),  $$T_q^+(x) < \Big(\frac{d^-(x)}{c-1}\Big)^{c-1} {{c-1}\choose {q+1}}\Big(\frac{d^+(x)}{d^-(x)}\Big)^{q+1}\frac{ (1+\beta)(q+1)}{c(1-\beta)-2q-2}.$$ 

Finally, observe that
$$ \begin{array}{lcl}
 \Big(\frac{d^-(x)}{c-1}\Big)^{c-1}\Big(\frac{d^+(x)}{d^-(x)}\Big)^{q+1} &=& \Big(\frac{d^-(x)}{c-1}\Big)^{c-1}\Big(\frac{d^+(x)}{c-1}\Big)^{q+1}\Big(\frac{c-1}{d^-(x)}\Big)^{q+1}\\ 
 &= &\Big(\frac{d^-(x)}{c-1}\Big)^{c-1-2q-2}\Big(\frac{d^+(x)d^-(x)}{(c-1)^2}\Big)^{q+1}.\end{array}$$
On the one hand, since $d^+(x) +d^-(x) = r(c-1)$, it follows that $\frac{d^+(x)d^-(x)}{(c-1)^2}\leq \frac{r^2(c-1)^2}{4(c-1)^2}= \frac{r^2}{4}$ and therefore $\Big(\frac{d^+(x)d^-(x)}{(c-1)^2}\Big)^{q+1}\leq \Big(\frac{r}{2}\Big)^{2q+2}$. 

On the other hand, $d^-(x)\leq \Delta(G_{r,c}) = \frac{r(c-1)+i_g(G_{r,c})}{2}= \frac{r(c-1)+r(c-1)\beta}{2}= \frac{r(c-1)(1+\beta)}{2}$ and therefore $\Big(\frac{d^-(x)}{c-1}\Big)^{c-1-2q-2} \leq \Big(\frac{r(1+\beta)}{2}\Big)^{c-1-2q-2}= \Big(\frac{r}{2}\Big)^{c-1-2q-2}(1+\beta)^{c-1-2q-2}$. Thus,  
$$\Big(\frac{d^-(x)}{c-1}\Big)^{c-1}\Big(\frac{d^+(x)}{d^-(x)}\Big)^{q+1}\leq \Big(\frac{r}{2}\Big)^{2q+2} \Big(\frac{r}{2}\Big)^{c-1-2q-2}(1+\beta)^{c-1-2q-2}=  \Big(\frac{r}{2}\Big)^{c-1}(1+\beta)^{c-1-2q-2}$$ and from here, the result follows.\end{proof}

\section{Maximal strong subtournament with  minimum degree at most $\left\lfloor \frac{c-2}{4}\right\rfloor+1$ }
Recall that if $T$  is a tournament of order $c$ such that $\delta(T) \geq \lfloor\frac{c-2}{4}\rfloor +1$,  then $T$ is strong. 
As an application of Theorem \ref{cota2}, we give  sufficient conditions for the existence of a maximal strong subtournament with  minimum degree at most $\left\lfloor \frac{c-2}{4}\right\rfloor+1$, in a balanced $c$-partite tournament.

 \begin{theorem}\label{fin}  Let $G_{r,c}$ be a balanced $c$-partite tournament, with $r\geq 2$,  such that  $\delta(G_{r,c}) \geq \lfloor\frac{c-2}{4}\rfloor\big(r+\mu(G_{r,c}) \big)\big(\frac{c-1}{c-2}\big)$.  Then  $G_{r,c}$ contains a strong connected tournament $T$ of order $c$ such that $\delta(T) \geq \lfloor\frac{c-2}{4}\rfloor +1$, whenever
 
\noindent {\it i)} $i_g(G_{r,c}) \leq  \frac{ r}{2}$  and $c \geq 13$, except for $14, 15$ and $18$.

   \noindent {\it ii)}  $i_g(G_{r,c}) \leq r $  and $c \geq 17$, except for $18, 19$ and $22$.

   \noindent {\it iii)} $i_g(G_{r,c}) \leq \frac{3r}{2} $  and $c \geq 21$, except for $22, 23$ and $26$.

 \end{theorem}
\begin{proof}  In order to prove this theorem, we first show the following.

\begin{claim}\label{formula} Let $r\geq 2$ and $c\geq 5$. For  every balanced $c$-partite tournament $G_{r,c}$ with  $\delta(G_{r,c}) \geq \lfloor\frac{c-2}{4}\rfloor\big(r+\mu(G_{r,c}) \big)\big(\frac{c-1}{c-2}\big)$ and 
 $i_g(G_{r,c}) \leq  \frac{\alpha r}{2}$ ($\alpha \geq 0$), if $$ 2^{c-2} > {{c-1}\choose {\lfloor\frac{c-2}{4}\rfloor+1}} \frac{ \Big(\frac{2c-2+\alpha}{2c-2}\Big)^{c-2-2\lfloor\frac{c-2}{4}\rfloor} (\lfloor\frac{c-2}{4}\rfloor+1)c }{c(\frac{2c-2-\alpha}{2c-2})-2\lfloor\frac{c-2}{4}\rfloor-2} $$
then  $G_{r,c}$ contains a strong connected tournament $T$ of order $c$ such that $\delta(T) \geq \lfloor\frac{c-2}{4}\rfloor +1.$
\end{claim}
Let  $G_{r,c}$ be a balanced $c$-partite tournament  as in the statement of the claim, and suppose there is no tournament $T$ of order $c$ in $G_{r,c}$ such that $\delta(T) \geq \lfloor\frac{c-2}{4}\rfloor +1.$  Thus, each of those tournaments has minimal degree at most $\lfloor\frac{c-2}{4}\rfloor$, and since there are $r^c$ tournaments of order $c$  in $G_{r,c}$,  it follows that $$\sum\limits_{x\in V(G_{r,c})} (T_{\lfloor\frac{c-2}{4}\rfloor}^+(x) + T_{\lfloor\frac{c-2}{4}\rfloor}^-(x)) \geq r^c.$$ Since $|V(G_{r,c})|=rc$, by Theorem \ref{cota2}, we  see that $$\Big(2rc\Big) {{c-1}\choose {\lfloor\frac{c-2}{4}\rfloor+1}}  \Big(\frac{r}{2}\Big)^{c-1}\frac{ \Big(1+\frac{i_g(G_{r,c})}{r(c-1)}\Big)^{c-2-2\lfloor\frac{c-2}{4}\rfloor} (\lfloor\frac{c-2}{4}\rfloor+1) }{c(1-\frac{i_g(G_{r,c})}{r(c-1)})-2\lfloor\frac{c-2}{4}\rfloor-2}\geq r^c ,$$ and since $i_g(G_{r,c}) \leq \frac{\alpha r}{2}$, it follows that  $\frac{i_g(G_{r,c})}{r(c-1)}\leq  \frac{\alpha}{2(c-1)}$; $1+ \frac{i_g(G_{r,c})}{r(c-1)}\leq  \frac{2c-2+\alpha}{2c-2}$  and  $1- \frac{i_g(G_{r,c})}{r(c-1)}\geq  \frac{2c-2-\alpha}{2c-2}$. Thus,  
$$\Big(2rc\Big) {{c-1}\choose {\lfloor\frac{c-2}{4}\rfloor+1}}  \Big(\frac{r}{2}\Big)^{c-1}\frac{ \Big( \frac{2c-2+\alpha}{2c-2}\Big)^{c-2-2\lfloor\frac{c-2}{4}\rfloor} (\lfloor\frac{c-2}{4}\rfloor+1) }{c( \frac{2c-2-\alpha}{2c-2})-2\lfloor\frac{c-2}{4}\rfloor-2}\geq r^c. $$
Multiplying both sides of the inequality by $\frac{2^{c-1}}{r^{c-1}}\Big(\frac{1}{2r}\Big)$, we obtain that $$ c{{c-1}\choose {\lfloor\frac{c-2}{4}\rfloor+1}} \frac{ \Big( \frac{2c-2+\alpha}{2c-2}\Big)^{c-2-2\lfloor\frac{c-2}{4}\rfloor} (\lfloor\frac{c-2}{4}\rfloor+1) }{c( \frac{2c-2-\alpha}{2c-2})-2\lfloor\frac{c-2}{4}\rfloor-2}\geq 2^{c-2}$$
and from here the claim follows. 

Let $f_\alpha(c) = {{c-1}\choose {\lfloor\frac{c-2}{4}\rfloor+1}} \frac{ \Big( \frac{2c-2+\alpha}{2c-2}\Big)^{c-2-2\lfloor\frac{c-2}{4}\rfloor} (\lfloor\frac{c-2}{4}\rfloor+1) c}{c( \frac{2c-2-\alpha}{2c-2})-2\lfloor\frac{c-2}{4}\rfloor-2}$ and $g(c) = 2^{c-2}$.  Notice that for $0\leq \alpha \leq \alpha'$, for every $c\geq 2$, $f_\alpha (c) \leq f_{\alpha'}(c)$.

If $i_g(G_{r,c}) \leq  \frac{ r}{2}$, it follows that $\alpha \leq 1$ and it is not hard to see that $f_1(c) < g(c)$ whenever $c\in\{13, 16, 19, 22\}$. Analogously,  if $i_g(G_{r,c}) \leq  r$, then $\alpha \leq 2$ and $f_2(c) < g(c)$ whenever $c\in\{17, 20, 23, 26\}$; and if $i_g(G_{r,c}) \leq  \frac{3r}{2}$, then $\alpha \leq 3$ and $f_3(c) < g(c)$ whenever $c\in\{21, 24, 27, 30\}$.

To end the proof, we just need to show that, for $\alpha \in\{1,2,3\}$, if for some $c\geq 13$ we have that  $f_\alpha(c) < g(c)$, then $f_\alpha(c+4) < g(c+4)$. For this we show that 
 $\frac{f_\alpha(c+4)}{f_\alpha(c)}\leq \frac{g(c+4)}{g(c)} $. 
 Clearly,  for every $c\geq 13$,  $\frac{g(c+4)}{g(c)}= 16$. On the other hand, it is not difficult to see that, for every $c\geq 13$,  
 $$\frac{ \Big( \frac{2(c+4)-2+\alpha}{2(c+4)-2}\Big)^{(c+4)-2-2\lfloor\frac{c+2}{4}\rfloor} }{ \Big( \frac{2c-2+\alpha}{2c-2}\Big)^{c-2-2\lfloor\frac{c-2}{4}\rfloor} } \leq  \Big( \frac{2c+6+\alpha}{2c+6}\Big)^{2}\leq \frac{6}{5},$$
 
and using a solver, it is possible to verify that,  for $c\geq 13$, 
 
$$\Bigg(
 \frac{{{(c+4)-1}\choose {\lfloor\frac{(c+4)-2}{4}\rfloor+1}}(\lfloor\frac{(c+4)-2}{4}\rfloor+1)(c+4)}{{{c-1}\choose {\lfloor\frac{c-2}{4}\rfloor+1}}(\lfloor\frac{c-2}{4}\rfloor+1)(c)} \Bigg) \Bigg(\frac{c( \frac{2c-2-\alpha}{2c-2})-2\lfloor\frac{c-2}{4}\rfloor-2}{(c+4)( \frac{2(c+4)-2-\alpha}{2(c+4)-2})-2\lfloor\frac{(c+4)-2}{4}\rfloor-2}\Bigg) \leq \frac{32}{3}.$$

Thus, for $c\geq 13$,  $\frac{f_\alpha(c+4)}{f_\alpha(c)}\leq \big(\frac{6}{5}\big)\big(\frac{32}{3}\big)< 16 = \frac{g(c+4)}{g(c)} $ and the result follows. \end{proof}

As we can observe from the proof of Claim \ref{formula},  it is possible to obtain analogous results to Theorem  \ref{fin}  for greater values of global irregularity.

\end{document}